\documentclass{amsart}

\chardef\bslash=`\\ % p. 424, TeXbook

\newcommand{\cl}{\operatorname{Cl}}

\newcommand{\diff}{\operatorname{\mathfrak{X}^1(M)}}

\newcommand{\per}{\operatorname{Per}}
\newcommand{\sink}{\operatorname{Sink}}
\newcommand{\sou}{\operatorname{Source}}

\newcommand{\sad}{\operatorname{Saddle}}

\newcommand{\di}{\operatorname{div}}

\newcommand{\dis}{\operatorname{Dis}}

\newcommand{\m}{\operatorname{Leb}}

\newcommand{\SR}{{\mathcal R}}

\newcommand{\SU}{{\mathcal U}}

\newcommand{\Crit}{\operatorname{Crit}}
\newcommand{\Sing}{\operatorname{Sing}}

\newcommand{\psad}{\operatorname{PSaddle}}

\newtheorem{theorem}{Theorem}[section] % 1st argument is your name for it
\newtheorem{lemma}[theorem]{Lemma}     % 2nd argument is what is printed
\newtheorem{corollary}[theorem]{Corollary}
\newtheorem{proposition}[theorem]{Proposition}
\newtheorem{defi}[theorem]{Definition}

\newtheorem{conjecture}[theorem]{Conjecture}

\newtheorem*{aldro}{Theorem A}
\newtheorem*{mata}{Theorem B}

\newtheorem*{hausdorff-limit}{Theorem E}

\newcommand{\eval}[2][\right]{\relax
  \ifx#1\right\relax \left.\fi#2#1\rvert}

\title[Existence of attractors for three-dimensional flows]{Existence of attractors for three-dimensional flows}
\author{C.A. Morales}

\address{Instituto de Matem\'atica, Universidade Federal do Rio de Janeiro, P. O. Box 68530, 21945-970 Rio
de Janeiro, Brazil.}
\email{morales@impa.br}

\thanks{Partially supported by CNPq, FAPERJ and PRONEX/DS from Brazil.}

\subjclass[2010]{Primary: 37D20; Secondary: 37C70}
\keywords{Singular-hyperbolic Attractor, Sink, Three-dimensional flow.}

\begin{document}

%\part{Use this type of header for very long papers only}
% use lowercase except for proper names
\maketitle
\section{Introduction} % use lowercase except for proper names
\label{intro}

\noindent
Araujo proved in the eighties that every $C^1$ generic surface diffeomorphism has either infinitely many attracting periodic orbits
(often called {\em sinks})
or else finitely many hyperbolic attractors whose basins form a full Lebesgue measure in the ambient manifold \cite{A}.
Therefore, hyperbolic attractors do exist for all $C^1$ generic surface diffeomorphisms (solving a question by Ren\'e Thom \cite{pp}).
It is natural to ask if these results hold for three-dimensional flows instead of surface diffeomorphisms.
Although this is true for nonsingular flows (as proved recently in \cite{ams}), the answer for this question is negative.
In fact, \cite{mp} used the geometric Lorenz attractor \cite{gw} in order to construct an open set of three-dimensional flows in the sphere $S^3$
for which there are no hyperbolic attractors.
But this last example
exhibits a singular-hyperbolic attractor (e.g. the geometric Lorenz one) whose basin has full Lebesgue measure.
It is then reasonable to ask if
Araujo's holds for three-dimensional flows
but replacing the hyperbolic attractor alternative
by the hyperbolic or singular-hyperbolic attractor one.

In this paper we will give positive answer for this last question.
More precisely,
we will prove that every $C^1$ generic three-dimensional flow
has either infinitely many sinks or finitely many hyperbolic or singular-hyperbolic attractors whose basins form a full Lebesgue measure set.
In particular,
every $C^1$ generic three-dimensional flow carries hyperbolic or singular-hyperbolic attractors.
The proof will use a recent result by Crovisier and Yang \cite{cy}.
Let us state our results in a precise way.

Hereafter, the term {\em three-dimensional flow} will be referred to a $C^1$ vector field on  a Riemannian compact connected boundaryless
three-dimensional manifold
$M$.
The corresponding space equipped with the $C^1$ vector field topology will be denoted by $\diff$.
The flow of $X\in\diff$ is denoted by $X_t$, $t\in\mathbb{R}$.

By a {\em critical point} of $X$ we mean a point $x$ satisfying
$X_t(x)=x$ for some $t\geq 0$. 
If this is satisfied for every $t\geq0$ we say that
$x$ is a {\em singularity}, otherwise it is a {\em periodic point}.
For every periodic point we have a minimal $t>0$ satisfying
$X_t(x)=x$. The minimal of such $t$'s is the period of $x$ denoted by
$t_x$ (or $t_{x,X}$ to indicate $X$). We denote by $\Crit(X)$ the set of critical points, by $\Sing(X)$ the set of singularities
and by $\per(X)$ the set of periodic points thus
$\Crit(X)=\Sing(X)\cup\per(X)$.
The {\em eigenvalues} of a critical point $x$ are either those of the linear automorphism $DX_{t_x}(x): T_xM\to T_xM$
not corresponding to the eigenvector $X(x)$ (periodic case) or those of $DX(x): T_xM\to T_xM$ (singular case).
We say that $x$ is a {\em sink} if its eigenvalues either are less than $1$ in modulus
(periodic case) or else with negative real part (singular case).
A {\em source} is a sink for the time-reversed flow $-X$.
The set of sinks and sources of $X$ will be denoted by $\sink(X)$ and $\sou(X)$ respectively.
A critical point is {\em hyperbolic} if it has no eigenvalues of modulus $1$ (periodic case) or
with zero real part (singular case).

For every point $x$ we define its {\em omega-limit set},
$$
\omega(x)=\left\{y\in M : y=\lim_{t_k\to\infty}X_{t_k}(x)\mbox{ for some integer sequence }t_k\to\infty\right\}.
$$
(If necessary we shall write $\omega_X(x)$ to indicate the dependence on $X$.)

We say that $\Lambda\subset M$ is
{\em invariant} if $X_t(\Lambda)=\Lambda$ for all $t\in\mathbb{R}$; and
{\em transitive} if there is $x\in\Lambda$ such that
$\Lambda=\omega(x)$.
The {\em basin} of any subset $\Lambda\subset M$ is defined by
$$
W^s(\Lambda)=\{y\in M : \omega(y)\subset \Lambda\}.
$$
(Sometimes we write $W^s_X(\Lambda)$ to indicate dependence on $X$).
An {\em attractor} is a transitive set $A$ exhibiting a neighborhood $U$ such that
$$
A=\displaystyle\bigcap_{t\geq0}X_t(U).
$$

A compact invariant set $\Lambda$ is {\em hyperbolic} if
there are a continuous $DX_t$-invariant tangent bundle decomposition
$T_\Lambda M=E_\Lambda^s\oplus E^X_\Lambda\oplus E_\Lambda^u$ over $\Lambda$
and positive numbers $K,\lambda$ such that $E^X_x$ is generated by $X(x)$, and for every $ (x,t)\in \Lambda\times\mathbb{R}^+$ we have
$$
\|DX_t(x)/E_x^s\|\leq Ke^{-\lambda t}
\mbox{ and } 
\|DX_{-t}(x)/E^u_{X_t(x)}\|\leq Ke^{-\lambda t},
$$
A {\em dominated splitting} for $X$ over an invariant set $I$
is a continuous tangent bundle $DX_t$-invariant splitting $T_I M=E_I\oplus F_I$
for which there are positive constants $K,\lambda$ satisfying
$$
\|DX_t(x)/E_x\|\cdot\|DX_{-t}(X_t(x))/F_{X_t(x)}\|\leq Ke^{-\lambda t},
\quad\quad\textrm{ for all }(x,t)\in I\times\mathbb{R}^+.
$$
We say that the dominating subbundle $E$ above is {\em contracting} if
$$
\|DX_t(x)/E_x\|\leq Ke^{-\lambda t},
\quad\quad\textrm{ for all } (x,t)\in I\times\mathbb{R}^+
$$
and that the central subbundle $F$ is {\em volume expanding} if
$$
|\det DX_t(x)/F_x|\geq K^{-1}e^{\lambda t},
\quad\quad\textrm{ for all } (x,t)\in I\times\mathbb{R}^+.
$$

A compact invariant set is {\em partially hyperbolic} if it has a dominated splitting with contracting
dominating direction.

A partially hyperbolic set $\Lambda$ is {\em singular-hyperbolic for $X$}
if the singularities of $X$ in $\Lambda$ are all hyperbolic and the central subbundle $F$ is volume expanding.
A {\em hyperbolic} (resp. {\em singular-hyperbolic}) {\em attractor (for $X$)} is an attractor which is simultaneously a hyperbolic
(resp. singular-hyperbolic) set for $X$.

We call $\mathcal{R}\subset \diff$ {\em residual} if it is a countable intersection of open and dense subsets.
We say that a {\em $C^1$ generic three-dimensional flow satisfies a certain property P} if
there is a residual subset $\mathcal{R}$ of $\diff$ such that $P$ holds for every element of $\mathcal{R}$.

With these definitions we can state our main result.

\begin{aldro}
\label{thaldro}
A $C^1$ generic three-dimensional flow has either infinitely many sinks or
finitely many hyperbolic or singular-hyperbolic attractors whose basins form a full Lebesgue measure set of $M$.
\end{aldro}

In particular, we obtain the existence of hyperbolic or singular-hyperbolic attractors
for $C^1$ generic three-dimensional flows:

\begin{corollary}
\label{coco1}
For every $C^1$ generic three-dimensional flow, there exists a hyperbolic or singular-hyperbolic attractor.
\end{corollary}

To prove Theorem A we will need the existence of a spectral decomposition of a certain invariant set.
To introduce it we will need some preliminars.

A critical point is a {\em saddle} if it has eigenvalues of modulus less and bigger than $1$ simultaneously
(periodic case) or with positive and negative real part simultaneously (singular case). The set of periodic saddles of $X$ is denoted by $\psad(X)$.

As is well known \cite{hps}, through any $x\in\psad(X)$ it passes a pair of invariant manifolds,
the so-called strong stable and unstable manifolds $W^{ss}(x)$ and $W^{uu}(x)$,
tangent at $x$ to the eigenspaces corresponding to the eigenvalue of modulus less and bigger than $1$ respectively.
Saturating these manifolds with the flow we obtain the stable and unstable manifolds $W^s(x)$ and $W^u(x)$ respectively.
A {\em homoclinic point} associated to $x$ is a point $q$ where $W^s(x)$ and $W^u(x)$ meet.
We say that $q$ is a {\em transverse homoclinic point} if $T_qW^s(x)\cap T_qW^u(x)$ is one-dimensional,
otherwise we call it {\em homoclinic tangency}.
The {\em homoclinic class} associated to $x$, denoted by $H(x)$, is the closure of the set of transverse homoclinic points associated to $x$.
We write $H_X(x)$ to indicate dependence on $X$.
By a homoclinic class we mean the homoclinic class associated to some saddle of $X$.
We denoted by $\cl(\cdot)$ the closure operator.

\begin{defi}
A non-empty subset $\mathcal{P}\subset \psad(X)$
is {\em homoclinically closed} if $H(p)\subset \cl(\mathcal{P})$ for every $p\in\mathcal{P}$.
\end{defi}

Basic examples of homoclinically closed subsets
are $\psad(X)$ itself and also the set $\psad_d(X)$ of {\em dissipative saddles},
i.e., those saddles for which the product of the eigenvalues is less than $1$ in modulus.
This follows from the Birkhoff-Smale Theorem \cite{hk}.

\begin{defi}
We say that a compact invariant set of $X$ {\em has a spectral decomposition} if
it is a finite disjoint union of transitive sets,
each one being either hyperbolic 
or a singular-hyperbolic attractor for either $X$ or $-X$.
\end{defi}

The following result will give a sufficient condition for existence of spectral decomposition for the closure
of homoclinically closed subsets of saddles.
Given $\Lambda\subset M$ we define
$\Lambda^*=\Lambda\setminus \Sing(X)$.
We define the vector bundle $N^X$ over $M^*$ whose fiber at $x\in M^*$ is the
the orthogonal complement of $X(x)$ in $T_xM$.

Denoting the projection $\pi_x: T_xM\to N_x^X$ we define the {\em Linear Poincar\'e flow} (LPF),
$P^X_t: N^X\to N^X$, by
$P^X_t(x)=\pi_{X_t(x)}\circ DX_t(x)$ whenever $t\in\mathbb{R}$.
We say that $\Lambda$ of $X$ {\em has a LPF-dominated splitting} if
$\Lambda^*\neq\emptyset$ and there exist a continuous tangent bundle decomposition $N^X=N^{s,X}\oplus N^{u,X}$ over $\Lambda^*$
with $dim N^{s,X}_x=dim N^{u,X}_x=1$ (for every $x\in\Lambda^*$) and $T>0$ such that 
$
\left\|P^X_{T}(x)/N^{s,X}_x\right\|\left\|P^X_{-T}(Y_T(x))/N^{u,X}_{X_T(x)}\right\|\leq\frac{1}{2}$, $\forall x\in\Lambda^*.$

With these definitions we obtain the following result.

\begin{mata}
\label{peorrito}
Let $X$ be a $C^1$ generic three-dimensional flow and
$\mathcal{P}\subset \psad(X)$ be homoclinically closed.
If $\cl(\mathcal{P})$ has a LPF-dominated splitting, then
$\cl(\mathcal{P})$ has a spectral decomposition.
\end{mata}

This result will be proved in the next section using the recent work \cite{cy} by Crovisier and Yang
(\footnote{Arbieto and Santiago \cite{as} will prove Theorem A and Corollary \ref{coco1} without appealing to \cite{cy}}).

Theorem B is strong enough to solve a conjecture in \cite{mp}. Indeed,
define the nonwandering set $\Omega(X)$ of a flow $X$ as the set of those points $x\in M$ such that for every
neighborhood $U$ of $x$ and every $T>0$ there is
$t>T$ such that $X_t(U)\cap U\neq\emptyset$. Clearly $\Omega(X)$ is a nonempty compact invariant set of $X$.
We say that $X$ is {\em singular-Axiom A} if $\Omega(X)$
has a spectral decomposition \cite{mp}.
In that a case we say that $X$ has no cycles if there are not
finitely many nonsingular orbits joining the pieces of the spectral decomposition in a cyclic way.
A {\em robustly singular-Axiom A flow} a flow
for which every nearby flow is singular-Axiom A.
Now we state the aforementioned conjecture in \cite{mp}
(see Conjecture 1.3 in p. 1577 of \cite{mp}):

\begin{conjecture}
\label{thC}
Every three-dimensional flow can be $C^1$ approximated by a flow exhibiting a homoclinic tangency or by a singular-Axiom A flow
without cycles.
\end{conjecture}

Let us prove this conjecture using Theorem B.

\vspace{5pt}

Define $R(M)$ as the (open) set of three-dimensional flows which cannot be $C^1$ approximated by
flows with a homoclinic tangency.
{\em The following sequence of assertions should be understood for $C^1$ generic $X\in R(M)$.}

As is well-known \cite{ah}, $\cl(\psad(X))$ has a LPF-dominated splitting.
Since $\psad(p)$ is homoclinically closed by the Birkhoff-Smale Theorem \cite{hk}, we conclude from Theorem B that
$\cl(\psad(X))$ has a spectral decomposition.
But from the arguments in \cite{PS} we see that
$$
(\cl(\sink(X))\setminus \sink(X))\cup(\cl(\sou(X))\setminus\sou(X))\subset\cl(\psad(X)).
$$
As $\cl(\psad(X))$ has spectral decomposition, there are only finitely many orbits of sinks or sources close to it.
This together with the previous inclusion implies
$\cl(\sink(X))\setminus \sink(X)=\cl(\sou(X))\setminus \sou(X)=\emptyset$, or, equivalently, that
$\sink(X)\cup\sou(X)$ consists of finitely many orbits.
On the other hand, we have that
$$
\Omega(X)=(\sink(X)\cup \sou(X))\cup \cl(\psad(X))
$$
by Pugh's General Density Theorem \cite{p}.
Since $\cl(\psad(X))$ has a spectral decomposition and $\sink(X)\cup\sou(X)$ consists of finitely many orbits, we conclude that
$X$ is singular-Axiom A.
The nonexistence of cycles was proved earlier
\cite{cmp}. This ends the proof.
\qed

\vspace{5pt}

Conjecture \ref{thC} is closely related to the recent result announced by Crovisier and Yang in \cite{cy}:
Every three dimensional flow can be $C^1$ approximated by
robustly singular hyperbolic flows, or by flows with a
homoclinic tangency
(they claim to have solved a conjecture by Jacob Palis \cite{pa}). Indeed,
we do not know if the approximation by singular-Axiom A flows in the conjecture can be performed by
robustly singular-Axiom A flows. We stress that, unlike Axiom A flows, the singular-Axiom A flows without cycles
need not be robustly singular-Axiom A in general \cite{mpu}.

\section*{Acknowledgements}
This paper is a joint work with professors A. Arbieto and B. Santiago which, unfortunately,
wanted to be excluded from the author list by some unmathematical reasons.
All of us are in debit with professor E.R. Pujals and R. Potrie by helpful conversations and, specially,
the last one by finding inaccuracies in an earlier version of this work.

\section{Proof of Theorem B}

\noindent
In this section we shall prove Theorem B.
For this we need some preliminary results.
We start with the following consequence of Lemma 3.1 in \cite{bgy}.

\begin{lemma}
\label{Q1}
Every compact invariant set without singularities but with a LPF-dominated splitting
of a $C^1$ generic three-dimensional flow is hyperbolic.
\end{lemma}

\begin{proof}
By Lemma 3.1 in \cite{bgy} we have that there is a residual subset $\mathcal{Q}_1$ of three-dimensional flows
for which every transitive set without singularities but with a LPF-dominated splitting is hyperbolic.
Fix $X\in\mathcal{Q}_1$ and a compact invariant set $\Lambda$ without singularities but with a LPF-dominated
splitting $N_\Lambda^X=N^{s,X}_\Lambda\oplus N^{u,X}_\Lambda$.
Suppose by contradiction that $\Lambda$ is not hyperbolic.
Then, by Zorn's Lemma, there is a minimally nonhyperbolic set $\Lambda_0\subset \Lambda$ (c.f. p.983 in \cite{PS}).
Assume for a while that $\Lambda_0$ is not transitive.
Then, $\omega(x)$ and $\alpha(x)=\omega_{-X}(x)$ are proper subsets of $\Lambda_0$, for every $x\in\Lambda_0$.
Therefore, both sets are hyperbolic and then we have
$$
\lim_{t\to\infty}\|P^X_t(x)/N^{s,X}_x\|=\lim_{t\to \infty}\|P^X_{-t}(x)/N^{u,X}_x\|=0,
\quad\textrm{ for all }
x\in \Lambda_0,
$$
which easily implies that $\Lambda_0$ is hyperbolic (see \cite{d}).
Since this is a contradiction, we conclude that $\Lambda_0$ is transitive.
As $X\in\mathcal{Q}_1$ and $\Lambda_0$ has a LPF-dominated splitting (by restriction), we conclude that $\Lambda_0$ is
hyperbolic, a contradiction once more proving the result.
\end{proof}

%The next two proposition will consider such limits containing singularities of index $2$.

Let $Y$ be a three-dimensional flow.
We say that $\sigma\in\Sing(Y)$ is {\em Lorenz-like for $Y$}
if its eigenvalues $\lambda_1,\lambda_2,\lambda_3$ are real and satisfy
$\lambda_2<\lambda_3<0<-\lambda_3<\lambda_1$ (up to some order).
The invariant manifold theory \cite{hps} asserts
the existence of {\em stable} and {\em unstable} manifolds denoted by
$W^{s}(\sigma)$, $W^{u}(\sigma)$ (or $W^{s,Y}(\sigma)$, $W^{u,Y}(\sigma)$ to emphasize $Y$)
tangent at $\sigma$ to the eigenvalues $\{\lambda_2,\lambda_3\}$ and $\lambda_1$ respectively.
There is an additional invariant manifold $W^{ss,Y}(\sigma)$, the {\em strong stable manifold}, contained in $W^{s,Y}(\sigma)$ and tangent at
$\sigma$ to the eigenspace corresponding to $\lambda_1$.

As in the remark after Lemma 2.13 in \cite{bgy} we obtain the following.

\begin{lemma}
\label{reslor}
If $X$ is a $C^1$ generic three-dimensional flow, then
every $\sigma\in\Sing(X)$ accumulated by periodic orbits is Lorenz-like, for either $X$ or $-X$, depending on whether $\sigma$ has three
real eigenvalues $\lambda_1,\lambda_2,\lambda_3$ satisfying either
$\lambda_2<\lambda_3<0<\lambda_1$ or
$\lambda_2<0<\lambda_3<\lambda_1$ (up to some order).
\end{lemma}

We shall use the following standard definitions.

%\vspace{0.1in}
%\emph{Hausdorff Limits}
%\vspace{0.1in}

\begin{defi}
The index $Ind(\sigma)$, of a singularity $\sigma$, is the number of eigenvalues with negative real part counted with multiplicity.
\end{defi}

\begin{defi}
The Hausdorff distance of two compact sets $A$ and $B$ is given by
$$d_h(A,B)=\max\left\{\sup_{x\in A}d(x,B),\sup_{y\in B}d(y,A)\right\}.$$
\end{defi}

It is well known that the space of compact subsets of $M$ is compact with this distance.
Clearly the Hausdorff limit of a sequence of periodic orbits is a nonempty compact invariant set.
Moreover, for $C^1$ generic flows the Hausdorff limit of sequences of periodic orbits
are characterized as the so-called chain-transitive sets \cite{c}, \cite{gy}.

The next result detects when a singularity is Lorenz-like through its index.

\begin{lemma}
\label{lor}
Let $H$ be the Hausdorff limit of a sequence of periodic orbits of a
$C^1$ generic three-dimensional flow $X$. If $H$ has a LPF-dominated splitting,
then every singularity $\sigma\in H\cap \Sing(X)$ satisfies one of the following:
\begin{enumerate}
\item
If $Ind(\sigma)=2$, then $\sigma$ is Lorenz-like for $X$ and
$
H\cap W^{ss,X}(\sigma)=\{\sigma\}.
$
\item
If $Ind(\sigma)=1$, then $\sigma$ is Lorenz-like for $-X$ and
$
H\cap W^{ss,-X}(\sigma)=\{\sigma\}.
$
\end{enumerate}
\end{lemma}

\begin{proof}
We only prove (1) because (2) is similar.
Since $H$ has a LPF-dominated splitting, we obtain from
Proposition 2.4 in \cite{d}
that $\sigma$ has three different real eigenvalues $\lambda_1,\lambda_2,\lambda_3$
satisfying $\lambda_2<\lambda_3<0<\lambda_1$ (up to some order). Since $\sigma\in H$,
we obtain that $\sigma$ is accumulated by periodic orbits.
Then, $\sigma$ is Lorenz-like for $X$ by Lemma \ref{reslor}.

To prove $H\cap W^{ss,X}(\sigma)=\{\sigma\}$ we assume by contradiction that this is not the case.
Then, there is $x\in H\cap W^{ss,X}(\sigma)\setminus\{\sigma\}$.
Set $H=\lim_{n\to\infty}O_n$ where each $O_n$ is a periodic orbit of $X$.

Choose sequences $x_n\in O_n$ and $t_n\to\infty$ such that $x_n\to x$
and $X_{t_n}(x_n)\to y$ for some $y\in W^{u,X}(\sigma)\setminus\{\sigma\}$.
Let $N^{s,X}\oplus N^{u,X}$ denote the LPF-dominated splitting of $H$.
Since $H$ is clearly connected, we can assume without loss of generality that this splitting is defined in the union
$\bigcup_n O_n$ (see Lemma 2.29 p.41 in \cite{ap}).

On the one hand, $N^{s,X}_x=N_x\cap W^{s,X}(\sigma)$ by Proposition 2.2 in \cite{d} and so
{\em $N^{s,X}_{x_n}$ tends to be tangent to $W^{s,X}(\sigma)$ at $x$ for $n$ large}.

On the other hand, Proposition 2.4 in \cite{d} says that $N^{s,X}_{y}$ is almost parallel to $E^{ss,X}_\sigma$,
and so, the directions $N^{s,X}_{X_{t_n}(x_n)}$ tends to be parallel to $E^{ss,X}_\sigma$.

Since $\lambda_2<\lambda_3$ and $N^{s,X}_{x_n}=P_{-t_n}(X_{t_n}(x_n))N^{s,X}_{X_{t_n}(x_n)}$,
we conclude that {\em $N^{s,X}_{x_n}$ tends to be
transversal to $W^{s,X}(\sigma)$ at $x$ for $n$ large}.

Since these two behaviors are contradictory, we obtain the result.
\end{proof}

Recall that a compact invariant set $\Lambda$ of a flow $X$ is {\em Lyapunov stable for $X$}
if for every neighborhood $U$ of $\Lambda$ there is a neighborhood $V\subset U$ of $\Lambda$ such that
$X_t(V)\subset U$, for all $t\geq 0$.

Let $\Lambda$ be a
compact invariant set with singularities (all hyperbolic) of $X$. We say
that {\em $\Lambda$ has dense singular unstable} (resp. {\em stable) branches} if for every $\sigma\in \Lambda\cap\Sing(X)$ one has
$\Lambda=\omega(q)$ (resp. $\Lambda=\omega_{-X}(q)$) for all $q\in W^u(\sigma)\setminus\sigma$ (resp. $q\in W^s(\sigma)\setminus\{\sigma\}$).

The results in \cite{cmp}, \cite{mp} imply the following lemma.

\begin{lemma}
\label{branched}
If $H$ is the Hausdorff limit of a sequence of periodic orbits of a
$C^1$ generic three-dimensional flow $X$, then the following alternatives hold
for every $\sigma\in H\cap \Sing(X)$:
\begin{enumerate}
\item
If $Ind(\sigma)=2$, then
$\cl(W^u(\sigma))$ is a Lyapunov stable set for $X$ with dense singular unstable branches and
$\cl(W^u(\sigma))=H$.
\item
If
$Ind(\sigma)=1$, then
$\cl(W^s(\sigma))$ is a Lyapunov stable set for $-X$ with dense singular stable branches
and $\cl(W^s(\sigma))=H$.
\end{enumerate}
\end{lemma}

Combining lemmas \ref{lor} and \ref{branched} we obtain the following result.

\begin{corollary}
\label{coro}
Let $H$ be the Hausdorff limit of a sequence of periodic orbits of a
$C^1$ generic three-dimensional flow $X$.
If $H$ has a LPF-dominated splitting, then one of the following alternatives hold:
\begin{enumerate}
\item
Every $\sigma\in H\cap \Sing(X)$ is Lorenz-like for $X$ and $H\cap W^{ss,X}(\sigma)=\{\sigma\}$.
\item
Every $\sigma\in H\cap \Sing(X)$ is Lorenz-like for $-X$ and $H\cap W^{ss,-X}(\sigma)=\{\sigma\}$.
\end{enumerate}
\end{corollary}

This permits us to separate the Hausdorff limits (of sequences of periodic orbits) with both singularities
and LPF-dominated splitting in two cases depending on whether there is a singularity of index $1$ or $2$.

Next we formulate the key result below by Crovisier and Yang.

\begin{theorem}[Theorem 1 in \cite{cy}]
\label{CY}
Let $\Gamma$ be a compact invariant set with a LPF-dominated splitting of a $C^3$ three-dimensional flow $Y$.
If
every periodic point in $\Gamma$ is hyperbolic saddle,
every $\sigma\in \Lambda\cap\Sing(Y)$ is Lorenz-like satisfying $W^{ss}(\sigma)\cap \Gamma=\{\sigma\}$ and
$\Gamma$ does not contain a minimal repeller whose dynamics is the suspension of an irrational rotation of the circle,
then $\Gamma$ is {\em dominated} (i.e. has a dominated splitting) for $Y$.
\end{theorem}

We shall use it to prove the following lemma.

\begin{lemma}
\label{basicLemma}
Let $\Lambda$ be a transitive set with a LPF-dominated splitting of a $C^1$ generic three-dimensional flow $X$.
If every singularity $\sigma\in \Lambda$ is Lorenz-like for $X$ satisfying $W^{ss}(\sigma)\cap \Lambda=\{\sigma\}$,
then $\Lambda$ is dominated for $X$.
\end{lemma}

\begin{proof}
Indeed, the result is obtained as in the proof of Lemma 3.1 in \cite{bgy}
with Theorem \ref{CY} playing the role of Theorem B in \cite{ah}.
We include details for the sake of completeness.
For this we need some basic definitions.

A compact invariant set $\Lambda$ is called {\em chain transitive} for a flow $X$ if for any $\epsilon>0$ and any $x,y\in \Lambda$
there are points $x_0,\cdots, x_n\in \Lambda$ and numbers $t_0,\cdots, t_{n-1}\in [1,\infty[$
such that $x_0=x$, $x_n=y$ and $d(X_{t_i}(x_i),x_{i+1})<\epsilon$ for all $0\leq i\leq n-1$.
For any $K\subset M$ we define $CR(X,K)$ as the set of those points $x$ for which there is a chain-transitive set $\Lambda$ satisfying
$x\in \Lambda\subset K$. This is a compact invariant set of $X$ contained in $K$.
We also define the {\em maximal invariant set} of $X$ in $K$:
$$
\max(X,K)=\displaystyle\bigcap_{t\in \mathbb{R}}X_t(K).
$$

Take a countable basis $\{U_n\}$ of $M$ and let $\mathcal{O}=\{O_n\}$ be such that each $\mathcal{O}_n$
is a finite union of elements of $\{U_n\}$.
For each $n$ we define
$$
\mathcal{D}_n=\{X\in \diff:CR(X,\cl(O_n)) \mbox{ is } \emptyset \mbox{ or dominated  for }X\},
$$
and
$$
\mathcal{N}_n=\{X\in\diff:CR(X,O_n)\mbox{ is neither dominated nor }\emptyset\}.
$$
By lemmas 2.9 (which is true for dominated sets instead of hyperbolic sets) and 2.10 in \cite{bgy} we have that
$\mathcal{D}_n\cup\mathcal{N}_n$ is open and dense in $\diff$.
It follows that
$$
\mathcal{G}=\displaystyle\bigcap_n(\mathcal{D}_n\cup\mathcal{N}_n)
$$
is residual in $\diff$. Let us prove that every $X\in \mathcal{G}$ satisfies the conclusion of the lemma.
Indeed, take $\Lambda$ as in the hypothesis of the lemma and suppose by contradiction that
$\Lambda$ is not dominated for $X$.

Take also $n$ such that
$\Lambda\subset O_n$ and a neighborhood $\mathcal{U}$ of $X$ such that, for every $Y\in\mathcal{U}$,
$\max(Y,\cl(O_n))$ has a LPF-dominated splitting for $Y$ and
every $\sigma\in \max(Y,\cl(O_n))$ is Lorenz-like satisfying
$$
W^{ss,Y}(\sigma)\cap\max(Y,\cl(O_n))=\{\sigma\}.
$$
Since $\Lambda$ is not dominated for $X$ (and $\emptyset\neq\Lambda\subset CR(X,\cl(O_n))$) we see that
$X\not\in \mathcal{D}_n$. As $X\in \mathcal{G}$, we conclude that $X\in\mathcal{N}_n$.
Now, take a $C^3$ Kupka-Smale flow $Y\in \mathcal{N}_n\cap \mathcal{U}$ having no
minimal repellers whose dynamics is the suspension of a irrational rotation of the circle
(the nonexistence of such dynamics is dense in any topology).

Since $Y\in\mathcal{N}_n$, one has that $CR(Y,\cl(O_n))$ is not dominated.
It follows from the definitions that $CR(Y,\cl(O_n))\cap (\sink(Y)\cup\sou(Y))$ consists of isolated orbits.
Therefore,
$$
\Gamma=CR(Y,\cl(O_n))\setminus (\sink(Y)\cup\sou(Y))
$$
is a compact (and obviously invariant) set of $Y$.
Since $CR(Y,\cl(O_n))$ is not dominated, we have that {\em $\Gamma$ is not dominated for $Y$}.

On the other hand, $Y\in\mathcal{U}$ thus $\max(Y,\cl(O_n))$ has a LPF-dominated splitting and, also, every $\sigma\in
Sing(Y)\cap \max(Y,\cl(O_n))$ is Lorenz-like satisfying $W^{ss,Y}(\sigma)\cap \max(Y,\cl(O_n))=\{\sigma\}$.
As $\Gamma\subset CR(Y,\cl(O_n))\subset\max(Y,\cl(O_n))$ we conclude the same for $\Gamma$ instead of $\max(Y,\cl(O_n))$.
Since $\Gamma$ has neither sinks nor sources, we conclude from Theorem \ref{CY} that {\em $\Gamma$ is dominated for $Y$}.
This is a contradiction so $\Lambda$ is dominated for $X$. The proof follows.
\end{proof}

The above lemma implies the following result.

\begin{proposition}
\label{thhausdorff-limit}
Let $H$ be the Hausdorff limit of a sequence of periodic orbits
of a $C^1$ generic three-dimensional flow $X$.
If $H$ has a LPF-dominated splitting, then
$H$ is a hyperbolic set (if $H\cap \Sing(X)=\emptyset$),
a singular-hyperbolic attractor for $X$ (if $H$ contains a singularity of index $2$)
or a singular-hyperbolic attractor for $-X$ (otherwise).
\end{proposition}

\begin{proof}
If $H\cap \Sing(X)=\emptyset$, then $H$ is hyperbolic by Lemma \ref{Q1}.
Now, suppose that $H$ contains a singularity $\sigma$ of index $2$.
Clearly,
$H$ is nontrivial (i.e. not equal to a single orbit) and by Corollary \ref{coro} we also have that it is the chain-recurrent class
of $\sigma$.
Since $H$ contains a singularity of index $2$, we have from the first alternative of Corollary \ref{coro}
that every $\sigma\in H\cap\Sing(X)$ is Lorenz-like and satisfies
$H\cap W^{ss,X}(\sigma)=\{\sigma\}$. Then, we can apply Lemma \ref{basicLemma} to conclude that $H$ has a dominated splitting for $X$.
Since $H$ is the chain recurrent class of $\sigma$ we conclude from Theorem C in \cite{gy} that
$H$ is a singular-hyperbolic attractor for $X$.
If $H$ contains a singularity of index $1$, then
the same argument with $-X$ instead of $X$ implies
that $H$ is a singular-hyperbolic attractor for $-X$.
This concludes the proof.
\end{proof}

%\begin{theorem}
%\label{peorrito}
%Let $X$ a $C^1$ generic three-dimensional flow and
%$\mathcal{P}\subset \psad(X)$ be homoclinically closed.
%If $\cl(\mathcal{P})$ has a LPF-dominated splitting, then
%$\cl(\mathcal{P})$ has a spectral decomposition.
%\end{theorem}

\begin{proof}[Proof of Theorem B]
Let $X$ a $C^1$ generic three-dimensional flow and
$\mathcal{P}\subset \psad(X)$ be homoclinically closed.
Suppose that $\cl(\mathcal{P})$ has a LPF-dominated splitting.

By taking the Hausdorff limit of sequences of periodic orbits in $\mathcal{P}$ accumulating on the singularities
of $X$ in $\cl(\mathcal{P})$ we obtain
from Proposition \ref{thhausdorff-limit} that
every $\sigma\in \cl(\mathcal{P})\cap\Sing(X)$ belongs to a singular-hyperbolic attractor
for either $X$ or $-X$.

Using that $\mathcal{P}$ is homoclinically closed we obtain
\begin{equation}
\label{spectral}
\cl(\mathcal{P})=
\cl\left(\displaystyle\bigcup\{H(p):p\in\mathcal{P}\}\right).
\end{equation}

We claim that the family $\{H(p):p\in\mathcal{P}\}$ is finite.
Otherwise, there is an infinite sequence $p_k\in\mathcal{P}$ yielding
infinitely many distinct homoclinic classes $H(p_k)$.
Consider the closure $\cl(\bigcup_k H(p_k))$, which is a compact invariant set contained in $\cl(\mathcal{P})$.
If this closure does not contain any singularity, then it would be a hyperbolic set by Lemma \ref{Q1}.
Since the number of homoclinic classes contained in any hyperbolic set is finite, we obtain a contradiction 
proving that $\cl(\bigcup_k H(p_k))$ contains a singularity $\sigma\in\cl(\mathcal{P})$.
But, as we have seem, any of these singularities belong to a singular-hyperbolic attractor for either $X$ or $-X$.
Since there are finitely many singularities, it must exist distinct $k,k'$ satisfying $H(p_k)=H(p_{k'})$.
But this is an absurd, so the claim follows.
Combining the claim with (\ref{spectral}) and the well-known fact that
the homoclinic classes are pairwise disjoint \cite{cmp} we obtain the desired spectral decomposition.
\end{proof}

\vspace{5pt}

\section{Proof of Theorem A}

\noindent
In this section we shall prove our main result.
We start with some useful definitions.

Let $X$ be a three-dimensional flow. Recall that a periodic point saddle if it has eigenvalues of modulus less and bigger than $1$ simultaneously.
Analogously for singularities by just replace $1$ by $0$ and the eigenvalues by their corresponding real parts.
Denote by $\sad(X)$ the set of saddles of $X$.

A critical point $x$ is {\em dissipative} if
the product of its eigenvalues (in the periodic case) or the divergence $\di X(x)$ (in the singular case)
is less than $1$ (resp. $0$).
Denote by $\Crit_d(X)$ the set of dissipative critical points. Define
the {\em dissipative region} by $\dis(X)=\cl(\Crit_d(X))$.

For every subset $\Lambda\subset M$ we define
$$
W^s_w(\Lambda)=\{x\in M:\omega(x)\cap \Lambda\neq\emptyset\}.
$$
(This is often called
{\em weak region of attraction} \cite{bs}.)

The following result was proved in \cite{ams} in the nonsingular case.
The proof in the general case is similar.

\begin{theorem}
\label{move-attractor}
Let $X$ be a $C^1$ generic three-dimensional flow.
Then, $W^s_w(\dis(X))$ has full Lebesgue measure.
\end{theorem}

Given a homoclinic class $H=H_X(p)$ of a three-dimensional flow $X$ we denote by
$H_Y=H_Y(p_Y)$ the continuation of $H$, where $p_Y$ is the analytic continuation of $p$
for $Y$ close to $X$ (c.f. \cite{pt}).

The following lemma was also proved in \cite{ams}. In its statement $\m$ denotes the normalized Lebesgue measure of $M$.

\begin{lemma}
\label{local}
If $X$ is a $C^1$ generic three-dimensional flow, then
for every hyperbolic homoclinic class $H$
there are an open neighborhood $\mathcal{O}_{X,H}$ of $f$ and a residual subset
$\mathcal{R}_{X,H}$ of $\mathcal{O}_{X,H}$ such that the following properties are equivalent:
\begin{enumerate}
\item
$\m(W^s_Y(H_Y))=0$ for every $Y\in\mathcal{R}_{X,H}$.
\item
$H$ is not an attractor.
\end{enumerate}
\end{lemma}

We also need the following lemma essentially proved in \cite{ao}.

\begin{lemma}
\label{AO}
If $X$ is a $C^1$ generic three-dimensional flow, then every singular-hyperbolic
attractor with singularities of either $X$ or $-X$ has zero Lebesgue measure.
\end{lemma}

\begin{proof}
Given $U\subset M$
we define $\mathcal{U}(U)$ as the set of flows $Y$ such that $\max(Y,U)$ is
a singular-hyperbolic set with singularities of $Y$.
We shall assume that $U$ is open.
It follows that $\mathcal{U}(U)$ is open in $\diff$.

Now define $\mathcal{U}(U)_n$ as the set of $Y\in \mathcal{U}(U)$ such that
$\m(\max(Y,U))<1/n$. It was proved in \cite{ao} that $\mathcal{U}(U)_n$ is open and dense in $\mathcal{U}(U)$.

Define $\SR(U)_n=\SU(U)_n\cup (\mathfrak{X}^1(M)\setminus\cl(\SU(U))$ which is open and dense set in $\mathfrak{X}^1(M)$.
Let $\{U_m\}$ be a countable basis of the topology, and $\{O_m\}$ be the set of finite unions of such $U_m$'s.
Define
$$
\mathcal{L}=\bigcap_m\bigcap_n\SR(O_m)_n.
$$
This is clearly a residual subset of three-dimensional flows. We can assume without loss of generality that
$\mathcal{L}$ is symmetric, i.e., $X\in\mathcal{L}$ if and only if $-X\in \mathcal{L}$.
Take $X\in \mathcal{L}$. Let $\Lambda$ be a singular-hyperbolic attractor for $X$.
Then, there exists $m$ such that $\Lambda=\Lambda_X(O_m)$.
Then $X\in \SU(O_m)$ and so $X\in \SU(O_m)_n$ for every $n$ thus $\m(\Lambda)=0$.
Analogously, since $\mathcal{L}$ is symmetric, we obtain that $\m(\Lambda)=0$ for every singular-hyperbolic attractor with singularities
of $-X$.
\end{proof}

Now we prove the following result which is similar to one in \cite{ams}
(we include its proof for the sake of completeness). In its statement
$\psad_d(X)$ denotes the set of periodic dissipative saddles of a three-dimensional flow $X$.

\begin{theorem}
\label{fui}
Let $Y$ be a $C^1$ generic three-dimensional flow. If $\cl(\psad_d(Y))$ has a spectral decomposition,
then every homoclinic class $H$ associated to a dissipative periodic saddle satisfying
$\m(W^s_Y(H))>0$ is either a hyperbolic attractor or a singular-hyperbolic attractor for $Y$.
\end{theorem}

\begin{proof}
Define the map $S:\diff\to 2^M_c$ by
$S(X)=\cl(\psad_d(X))$.
This map is clearly lower-semicontinuous, and so, upper semicontinuous in a residual subset
$\mathcal{N}$ (for the corresponding definitions see \cite{k1}, \cite{k}).

By the flow-version of the main result in \cite{a1}, there is a residual subset $\mathcal{R}_7$ of three-dimensional flows
$X$ such that for every singular-hyperbolic attractor $C$ for $X$ (resp. $-X$) there are neighborhoods
$U_{X,C}$ of $C$, $\mathcal{U}_{X,C}$ of $X$ and a residual subset $\mathcal{R}^0_{X,C}$ of $\mathcal{U}_{X,C}$ such that
for all $Y\in\mathcal{R}^0_{X,C}$ if $Z=Y$ (resp. $Z=-Y$) then

\begin{equation}
\label{considera}
C_Y=\displaystyle\bigcap_{t\geq0}Z_t(U_{X,C})
\mbox{ is a singular-hyperbolic attractor for }Z.
\end{equation}

Define $\mathcal{R}= \mathcal{N}\cap\mathcal{R}_7$.
Clearly $\mathcal{R}$ is a residual subset of three-dimensional flows.
Define
$$
\mathcal{A}=\{f\in \mathcal{R}:\cl(\psad_d(X))\mbox{ has no spectral decomposition}\}.
$$

Fix $X\in\mathcal{R}\setminus \mathcal{A}$.
Then, $X\in\mathcal{R}$ and
$\cl(\psad_d(X))$ has a spectral decomposition
$$
\cl(\psad_d(X))=\left(\displaystyle\bigcup_{i=1}^{r_X}H^i\right)\cup\left(\displaystyle\bigcup_{j=1}^{a_X} A^j\right)
\cup\left(\displaystyle\bigcup_{k=1}^{b_X} R^k\right)
$$
into hyperbolic homoclinic classes $H_i$
($1\leq i\leq r_X$), singular-hyperbolic attractors $A^j$ for $X$ ($1\leq j\leq a_X$), and
singular-hyperbolic attractors $R^k$ for $-X$ ($1\leq k\leq b_X$).

As $X\in \mathcal{R}_7$, we can consider for each $1\leq i\leq r_X$, $1\leq j\leq a_X$ and $1\leq k\leq b_X$ the neighborhoods
$\mathcal{O}_{X, H^i}$, $\mathcal{U}_{X,A^j}$ and $\mathcal{U}_{X,R^k}$ of $X$ as well as their residual subsets $\mathcal{R}_{X,H^i}$,
$\mathcal{R}^0_{X,A^j}$ and $\mathcal{R}^0_{X,R^k}$ given by Lemma \ref{local} and (\ref{considera}) respectively.

Define
$$
\mathcal{O}_X=\left(\displaystyle\bigcap_{i=1}^{r_X}\mathcal{O}_{X,H^i}\right)\cap
\left(\displaystyle\bigcap_{j=1}^{a_X}\mathcal{U}_{X,A^j}\right)\cap \left(\displaystyle\bigcap_{k=1}^{b_X}\mathcal{U}_{X,R^k}\right)
$$
and
$$
\mathcal{R}_X=\left(\displaystyle\bigcap_{i=1}^{r_X}\mathcal{R}_{X,H^i}\right)\cap
\left(\displaystyle\bigcap_{j=1}^{a_X}\mathcal{R}_{X,A^j}^0\right)\cap \left(\displaystyle\bigcap_{k=1}^{b_X}\mathcal{R}_{X,R^k}^0\right).
$$
Clearly $\mathcal{R}_X$ is residual in $\mathcal{O}_X$.

From the proof of Lemma \ref{local} in \cite{ams} we obtain for each $1\leq i\leq r_X$
a compact neighborhood $U_{X,i}$ of $H^i$ such that
\begin{equation}
\label{toloo}
H^i_Y=\displaystyle\bigcap_{t\in\mathcal{R}}Y_t(U_{X,i})\mbox{ is hyperbolic and equivalent to }H^i,
\textrm{ for all } Y\in \mathcal{O}_{Y,H^i}.
\end{equation}
As $X\in\mathcal{N}$, $S$ is upper semicontinuous at $X$ so we can further assume that
$$
\cl(\psad_d(Y))\subset \left(\displaystyle\bigcup_{i=1}^{r_X} U_{X,i}\right)\cup
\left(\displaystyle\bigcup_{j=1}^{a_X} U_{X,A^j}\right)\cup \left(\displaystyle\bigcup_{k=1}^{b_X} U_{X,R^k}\right),
\,\textrm{ for all }Y\in\mathcal{O}_{X}.
$$
It follows that
\begin{equation}
\label{tolo}
\cl(\psad_d(Y))=\left(\displaystyle\bigcup_{i=1}^{r_X}H^i_Y\right)\cup
\left(\displaystyle\bigcup_{j=1}^{a_X}A^j_Y\right)\cup \left(\displaystyle\bigcup_{k=1}^{b_X}R^k_Y\right),
\,\textrm{ for all }Y\in\mathcal{R}_X.
\end{equation}

Next we
take a sequence $X^i\in\mathcal{R}\setminus \mathcal{A}$ which is dense
in $\mathcal{R}\setminus \mathcal{A}$.

Replacing $\mathcal{O}_{X^i}$ by
$\mathcal{O}'_{X^i}$ where
$$
\mathcal{O}'_{X^0}=\mathcal{O}_{X^0}
\mbox{ and }
\mathcal{O}'_{X^i}=\mathcal{O}_{X^i}\setminus\left(\displaystyle\bigcup_{j=0}^{i-1}\mathcal{O}_{X^j}\right), \mbox{ for } i\geq1,
$$
we can assume that the collection
$\{\mathcal{O}_{X^i}:i\in\mathbb{N}\}$ is pairwise disjoint.

Define
$$
\mathcal{O}_{12}=\displaystyle\bigcup_{i\in\mathbb{N}}\mathcal{O}_{X^i}
\quad\mbox{ and }
\quad
\mathcal{R}'_{12}=\displaystyle\bigcup_{i\in\mathbb{N}}\mathcal{R}_{X^i}.
$$

We claim that $\mathcal{R}'_{12}$ is residual in $\mathcal{O}_{12}$.

Indeed, for all $i\in\mathbb{N}$ write $\mathcal{R}_{X^i}=\displaystyle\bigcap_{n\in \mathbb{N}} \mathcal{O}^n_i$,
where $\mathcal{O}^n_i$ is open-dense in
$\mathcal{O}_{X^i}$ for every $n\in\mathbb{N}$.
Since
$\{\mathcal{O}_{X^i}:i\in\mathbb{N}\}$ is pairwise disjoint, we obtain
$$
\displaystyle\bigcap_{n\in\mathbb{N}}
\displaystyle\bigcup_{i\in\mathbb{N}}\mathcal{O}^n_{i}
\subset
\displaystyle\bigcup_{i\in\mathbb{N}}
\displaystyle\bigcap_{n\in\mathbb{N}}\mathcal{O}^n_{i}=\displaystyle\bigcup_{i\in\mathbb{N}}\mathcal{R}_{X^i}=\mathcal{R}'_{12}.
$$
As $\displaystyle\bigcup_{i\in\mathbb{N}}\mathcal{O}^n_{X^i}$
is open-dense in $\mathcal{O}_{12}$, $\forall n\in \mathbb{N}$,
we obtain the claim.

Finally we define
$$
\mathcal{R}_{11}=\mathcal{A}\cup \mathcal{R}'_{12}.
$$
Since $\mathcal{R}$ is a residual subset of three-dimensional flows, we conclude as in Proposition 2.6 of \cite{m}
that $\mathcal{R}_{11}$ is also a residual subset of three-dimensional flows.

Take $Y\in\mathcal{R}_{11}$ such that $\cl(\psad_d(Y))$ has a spectral decomposition
and let $H$ be a homoclinic class associated to a dissipative saddle of $Y$.
Then,
$H\subset \cl(\psad_d(Y))$ by Birkhoff-Smale's Theorem \cite{hk}.

Since $\cl(\psad_d(Y))$ has a spectral decomposition, we have $Y\notin \mathcal{A}$ so $Y\in \mathcal{R}'_{12}$
thus $Y\in\mathcal{R}_X$ for some $X\in\mathcal{R}\setminus \mathcal{A}$.
As $Y\in\mathcal{R}_X$, (\ref{tolo}) implies
$H=H^i_Y$ for some $1\leq i\leq r_X$ or $H=A^j_Y$ for some $1\leq j\leq a_X$ or $H=R^k_Y$ for some $1\leq k\leq b_X$.

Now, suppose that $\m(W^{s}_Y(H))>0$. Since $Y\in\mathcal{R}_X$, we have $Y\in \mathcal{R}^0_{X,R^k}$ for all $1\leq k\leq b_X$.
As $W^s_Y(R^k_Y)\subset R^k_Y$ for every $1\leq k\leq b_X$, we conclude by Lemma \ref{AO} that $H\neq R^k_Y$ for
every $1\leq k\leq b_X$.

If $H=A^j_Y$ for some $1\leq j\leq a_X$ then
$H$ is an attractor and we are done. Otherwise,
$H=H^i_Y$ for some $1\leq i\leq r_X$. As $Y\in\mathcal{R}_X$, we have
$Y\in \mathcal{R}_{X,H^i}$ and, since $f\in \mathcal{R}_{12}$, we conclude from Lemma \ref{local}
that $H^i$ is an attractor. But by (\ref{toloo}) we have that $H^i_Y$ and $H^i$ are equivalent, so, $H^i_Y$ is an attractor too and we are done.
\end{proof}

\begin{proof}[Proof of Theorem A]
Let $X$ be a $C^1$ generic three-dimensional flow with only a finite number of sinks. Then,
we can prove as in \cite{ams} or \cite{PS} that
$\cl(\psad_d(X))$ has a LPF-dominated splitting.
Since $\psad_d(X)$ is homoclinically closed by the Birkhoff-Smale Theorem \cite{hk},
we conclude from Theorem B that $\cl(\psad_d(X))$ has a spectral decomposition.

Since $X$ is $C^1$ generic we can assume that $X$ is Kupka-Smale too.
Then, we have the following decomposition:
$$
\dis(X)
=\cl(\sad_d(X)\cap\Sing(X))\cup\cl(\psad_d(X))\cup\sink(X)
$$
yielding
$$
W^s_w(\dis(X))=
\left(\bigcup\{W^s(\sigma):\sigma\in\sad_d(X)\cap\Sing(X)\mbox{ and }W^s_w(\sigma)=W^s(\sigma)\}\right)
$$
$$
\cup
\left(\bigcup\{W^s_w(\sigma):\sigma\in\sad_d(X)\cap\Sing(X)\mbox{ and }W^s_w(\sigma)\neq W^s(\sigma)\}\right)\cup
$$
$$
W^s_w(\cl(\psad_d(X)))\cup W^s(\sink(X)).
$$

One can esasily check that the first element in the above union has zero measure.

Moreover, by the Hayashi's Connecting Lemma \cite{h},
we can assume without loss of generality that every $\sigma\in\sad_d(X)\cap\Sing(X)$ satisfying $W^s_w(\sigma)\neq W^s(\sigma)$ belongs to
$\cl(\psad_d(X))$.

Since $W^s_w(\dis(X))$ has full Lebesgue measure by Theorem \ref{move-attractor}, we conclude that
$$
\m(W^s_w(\cl(\psad_d(X)))\cup W^s(\sink(X)))=1.
$$
But, $\cl(\psad_d(X))$ has a spectral decomposition
$$
\cl(\psad_d(X))=\displaystyle\bigcup_{i=1}^rH_i
$$
into finitely many disjoint homoclinic classes $H_i$, $1\leq i\leq r$,
each one being either hyperbolic (if $H_i\cap\Sing(X)=\emptyset$)
or a singular-hyperbolic attractor for either $X$ or $-X$ (otherwise). Replacing above we obtain
$$
\m\left(\left(\displaystyle\bigcup_{i=1}^rW^s_w(H_i)\right)\cup W^s(\sink(X))\right)=1.
$$
Now the results in Section 3 of \cite{cmp}
imply that each $H_i$ can be written as
$H_i=\Lambda^+\cap \Lambda^-$,
where $\Lambda^\pm$ is a Lyapunov stable set for $\pm X$.
We conclude from Lemma 2.2 in \cite{cmp} that $W^s_w(H_i)=W^s(H_i)$ thus
$$
\m\left(\left(\displaystyle\bigcup_{i=1}^rW^s(H_i)\right)\cup W^s(\sink(X))\right)=1.
$$
Let $1\leq i_1\leq \cdots \leq i_d\leq r$ be such that $\m(W^s(H_{i_k}))>0$ for every $1\leq k\leq d$.
By Theorem \ref{fui} we have that each $H_{i_k}$ for $1\leq k\leq d$ 
is either a hyperbolic attractor or a singular-hyperbolic attractor for $X$.

As the basins of the remainder homoclinic classes in the collection
$H_1,\cdots,H_r$ are negligible, we can remove them from the above union yielding
$$
\m\left(
\left(\displaystyle\bigcup_{k=1}^dW^s(H_{i_k})\right)\cup W^s(\sink(X))
\right)=1.
$$
Since each $H_{i_k}$ is a hyperbolic or singular-hyperbolic attractor for $X$
and $\sink(X)$ consists of finitely many orbits, we are done.
\end{proof}

\end{document}